\documentclass[11pt]{amsart}

\usepackage[utf8]{inputenc}
\usepackage{amsthm, textcomp, amsmath, amsfonts,  amssymb }
\usepackage[pdftex]{graphicx}
\usepackage{url}
\usepackage{subfig}
\usepackage[colorlinks]{hyperref}
\usepackage[a4paper,left=2cm,right=2cm]{geometry}
\allowdisplaybreaks[4]

\DeclareMathOperator{\M}{\mathcal{M}}

\DeclareMathOperator{\R}{\mathbb{R}}

\begin{document}

\title%[Almost positive kernels]
{Almost positive kernels on compact Riemannian manifolds}

%\author[L. Brandolini]{Luca Brandolini}
%\address[L. Brandolini]{Dipartimento di Ingegneria Gestionale, dell'Informazione e della Produzione,
%Universit\`a degli Studi di Bergamo, Viale Marconi 5, Dalmine BG, Italy}
%\email{luca.brandolini@unibg.it}

\author[B. Gariboldi]{Bianca Gariboldi}
\address[B. Gariboldi]{Dipartimento di Ingegneria Gestionale, dell'Informazione e della Produzione,
Universit\`a degli Studi di Bergamo, Viale Marconi 5, Dalmine BG, Italy}
\email{biancamaria.gariboldi@unibg.it}

\author[G. Gigante]{Giacomo Gigante}
\address[G. Gigante]{Dipartimento di Ingegneria Gestionale, dell'Informazione e della Produzione,
Universit\`a degli Studi di Bergamo, Viale Marconi 5, Dalmine BG, Italy}
\email{giacomo.gigante@unibg.it}

\begin{abstract}
We show how to build a kernel  
\[
K_X(x,y)=\sum_{m=0}^Xh(\lambda_m/{\lambda_X})\varphi_m(x)\overline{\varphi_m(y)}
\]
 on a compact Riemannian manifold $\M$,
 which is positive up to a negligible error and
such that $K_X(x,x)\approx X$. Here $0=\lambda_0^2\le\lambda_1^2\le\ldots$ are the eigenvalues of the Laplace-Beltrami operator on $\M$, listed with repetitions, and $\varphi_0,\,\varphi_1,\ldots$ an associated system of eigenfunctions, forming an orthonormal basis of $L^2(\M)$.
The function $h$ is smooth up to a certain minimal degree, even, compactly supported in $[-1,1]$ with $h(0)=1$, and $K_X(x,y)$ turns out to be an approximation to the identity.

\end{abstract}



\thanks{All authors have been supported by an Italian GNAMPA 2020 project. They also wish to thank Luca Brandolini, Leonardo Colzani and Giancarlo Travaglini for several useful conversations on the subject of the paper.}

\maketitle

\newtheorem{theorem}{Theorem}
\newtheorem{corollary}[theorem]{Corollary}
\newtheorem{lemma}[theorem]{Lemma}
\newtheorem{definition}[theorem]{Definition}
\newtheorem{prop}[theorem]{Proposition}
\newtheorem{ex}[theorem]{Example}
\newtheorem{remark}[theorem]{Remark}

%\setcounter{tocdepth}{1}
%\tableofcontents

\section{Introduction}

Let $\mathcal C\subset\R^d$ be convex and symmetric 
and define the trigonometric polynomial
\begin{align*}\label{trigonometric}
T_{\mathcal C}(x)&=\frac{1}{\mathrm{card}((\frac 12 \mathcal C)\cap \mathbb Z^d)}\sum_{ \ell,  k \in \frac12 \mathcal C}e^{2\pi i ( \ell- k)\cdot  x}=\frac{1}{\mathrm{card}((\frac 12 \mathcal C)\cap \mathbb Z^d)}\left|\sum_{ \ell \in \frac12 \mathcal C}e^{2\pi i  \ell\cdot  x}\right|^2\\
&=\sum_{m\in\mathcal C}\frac{\mathrm{card}(\frac 12\mathcal C\cap(\frac12\mathcal C+m)\cap\mathbb Z^d)}{\mathrm{card}((\frac 12 \mathcal C)\cap \mathbb Z^d)}e^{2\pi i m\cdot  x}.
\end{align*}
The above identities immediately show that $T_{\mathcal C}(x)\ge0$, that its Fourier coefficients vanish outside $\mathcal C$, that $\widehat T(0)=1$, and that $T_\mathcal C(0)=\mathrm{card}( \frac 12\mathcal C\cap \mathbb Z^d)$.

In particular, when $\mathcal C$ is the axis-parallel, symmetric box of sides $2Y_1,\ldots,2Y_d$, then $T_\mathcal C $ is just the standard $d$-dimensional Fej\'er kernel
\[
F_{2\lfloor Y_1/2\rfloor+1}(x_1)\ldots F_{2\lfloor Y_d/2\rfloor+1}(x_d),
\]    
where
\[
F_n(t)=\sum_{-n\le m\le n}\left(1-\frac {|m|}n\right)e^{2\pi i mt}.
\]
One could also let $\mathcal C$ be the ball centered at the origin and with radius $Y$, and in this case the non-vanishing Fourier coefficients of $T_{\mathcal C}$ are just those corresponding to the eigenvalues $4\pi^2 |m|^2$ of the (positive) Laplace-Beltrami operator on the torus which are smaller than or equal to $4\pi^2 Y^2$. Notice that in this case, there are $\approx Y^d$ such eigenvalues, and that since $T_\mathcal C(0)=\mathrm{card}( \frac12\mathcal C\cap \mathbb Z^d)\approx Y^d$, then $T_\mathcal C(0)$ is essentially the number of eigenvalues less than or equal to $4\pi Y^2$.

Let now $\left(  \mathcal{M},g\right)  $ be a $d$-dimensional compact connected
Riemannian manifold, where the Riemannian distance $d\left(x,y\right)$ and the Riemannian measure are normalized so that the total measure of $\mathcal{M}$ equals $1$. Let
$\left\{  \lambda_{m}^{2}\right\}  _{m=0}^{+\infty}$ be the sequence of
eigenvalues of the (positive) Laplace-Beltrami operator $\Delta$, listed in
increasing order with repetitions, and let $\left\{  \varphi_{m}\right\}
_{m=0}^{+\infty}$ be an associated sequence of orthonormal eigenfunctions. In
particular $\varphi_{0}\equiv1$ and $\lambda_{0}=0$. This allows to define the
Fourier coefficients of $L^{1}(\mathcal{M})$ functions as
\[
\widehat{f}(\lambda_{m})=\int_{\mathcal{M}}f(x)\overline{\varphi_{m}(x)}dx,
\]
where the integration is with respect to the Riemannian measure, and the associated Fourier series
\[
\sum_{m=0}^{+\infty}\widehat{f}(\lambda_{m})\varphi_{m}(x).
\]
We would like to extend the construction of the above type of kernel to the case of Riemannian manifolds. In particular it would be very interesting to have a kernel of the form 
\begin{equation}
\label{kappaX}
K_X(x,y)=\sum_{m=0}^X a(\lambda_m,{\lambda_X})\varphi_m(x)\overline{\varphi_m(y)}
\end{equation}
which is nonnegative and such that $a(0,\lambda_X)=1$, and $K_X(x,x)\gtrsim X$. If possible, it would be great to have 
$0\le a(\lambda_m,\lambda_X)\le 1$.

Observe that by Weyl's estimates, $X$ is essentially the number of eigenvalues $\lambda_m^2$ that are smaller than or equal to $\lambda_X^2$ (and this number is essentially $\lambda_X^{d}$). Thus, this type of kernel could be considered as a generalization to the case of manifolds of the kernel $T_{\mathcal C}$ defined above, when $\mathcal C$ is the ball of radius $Y\approx X^{1/d}\approx \lambda_X$. 

We do not know if this type of kernels in a general manifold exist.
G. Travaglini \cite{giancarlo} proved that one can define certain Fej\'er kernels on
compact Lie groups which are nonnegative.   
Furthermore, it is easy to see that, in a compact two-point homogeneous space, if a kernel has finite spectrum, then also its square (which is nonnegative) has finite spectrum and a suitable normalization has mean one. In particular, R. Askey \cite{askey} showed that the kernels corresponding to certain Ces\`aro means are positive in certain compact two-point homogeneous spaces, and conjectured their positivity in all such spaces (see also \cite{AG}).

The first natural choice that comes to mind when in need of one such kernel is the heat kernel 
\[
p_t(x,y)=\sum_{m=0}^{+\infty}e^{-\lambda_m^2t}\varphi_m(x)\overline{\varphi_m(y)}, \quad t>0.
\]
It is well known that the above heat kernel is positive, all the coefficients are clearly between $0$ and $1$, the coefficient corresponding to $\lambda_0^2$ equals $1$, and $p_t(x,x)\approx t^{-d/2}$ for small $t$. The only problem with it is therefore that the coefficients do not vanish for $m>t^{-d/2}$. It can be proved (see \cite{BDS}) that
\[
\left|\sum_{m=X+1}^{+\infty}e^{-\lambda_m^2t}\varphi_m(x)\overline{\varphi_m(y)}\right|\lesssim t^{-d+1/2}(X^{2/d}t)^{d-3/2}e^{-X^{2/d}t}.
\] 
Thus, setting $t=cX^{-2/d}\log X$, the kernel
\[
\widetilde p_t(x,y)=\sum_{m=0}^{X}e^{-\lambda_m^2t}\varphi_m(x)\overline{\varphi_m(y)}=p_t(x,y)+O(X^{2-c-1/d}/\log X)
\]
is positive up to the remainder $O(X^{2-c-1/d}/\log X)$, all its Fourier coefficients vary between $0$ and $1$, the coefficient corresponding to $\lambda_0^2$ equals $1$, but $\widetilde p_t(x,x)\approx p_t(x,x)\approx X/\log^{d/2}X$. This strategy therefore gives a good estimate of the remainder, uniform in the variables $x$ and $y$, but generates a logarithmic loss in the diagonal estimate of the kernel.
Observe that the choice $t=cX^{-2/d}$ would give 
$p_t(x,x)\approx X$, as desired, but the remainder 
would be too big, precisely $O(X^{2-1/d})$.

Throughout the paper, we will denote $\mathcal F_d$ the $d$-dimensional Fourier transform
\[
\mathcal F_d f(\xi)=\int_{\R^d}f(x)e^{-2\pi i x\cdot\xi}dx
,\]
and when $f$ will be radial, with a slight abuse of notation, by ${\mathcal F_d}f(r)$ we will mean ${\mathcal F_d}f(z)$ for all those $z\in\R^d$ such that $|z|=r$. We will also denote $\mathcal{C}$ the cosine transform 
\[\mathcal{C}f(s)=\int_{\mathbb{R}} f(t)\cos(st)dt \]
and its inverse (on even functions) $\mathcal{C}^{-1}$ by
\[ \mathcal C^{-1} f(t)=\dfrac{1}{2\pi}\int_{\R} f(s)\cos(st)ds=\dfrac{1}{2\pi}\mathcal C f(t).\]

Our main result is the following
\begin{theorem}\label{main}
(i) There exists a nonnegative function $\alpha_0\in\mathcal C^{\infty}(\M\times\M)$, with $\alpha_0(x,x)=1$ such that the following holds.
Let $h$ be an integrable radial function on $\R^d$, compactly supported in the ball centered at the origin and with radius $1$ and such that for $G>d+1$ and for some positive constant $C$,
\[
\left|\mathcal F_dh(t)\right|\le C\frac 1{(1+t)^{2G}}.
\]
Then
\begin{align*}
K_X(x,y)&:=\sum_{m=0}^Xh\left(\frac{\lambda_m}{\lambda_X}\right)\varphi_m(x)\overline{\varphi_m(y)}\\
&=\frac{\alpha_0(x,y)}{(2\pi)^d}\lambda_X^d {\mathcal F_d}h\left(\frac{\lambda_Xd(x,y)}{2\pi}\right)
+O\left(\frac{\lambda_X^{d-2}}{\left(1+\lambda_Xd(x,y)\right)^{\lfloor 2G\rfloor-2d-2}}\right).
\end{align*}
(ii) For any integer $G$ there exists a (non vanishing) integrable radial function $h$ defined in $\mathbb R^d$, compactly supported in the unit ball, such that $0\le h(x)\le h(0)=1$ for all $x$, and such that for all $t$
\[
0\le\mathcal F_dh(t)\le C\frac{1}{(1+t)^{2G}}.
\]
\end{theorem}

Point (ii) of Theorem \ref{main} is in fact trivial (see the proof at the end of Section \ref{proofs}). 
The function $h$ in point (ii) satisfies all the hypotheses of point (i).
Furthermore, with this choice of $h$, $K_X$ has all the properties we mentioned after equation \eqref{kappaX}, and non-negativity up to the remainder. In particular $h(\lambda_0/\lambda_X)=h(0)=1$, $K_X(x,x)=\lambda_X^d\mathcal F_dh(0)/(2\pi)^d +O(\lambda_{X}^{d-2})\approx X$ and $0\le h(\lambda_m/\lambda_X)\le h(0)=1$. 

Here we can observe that a kernel $K_X$ as in Theorem \ref{main} is in fact an approximation to the identity, when $G$ is sufficiently large.

\begin{corollary}\label{cor}
Let $h$ be an integrable radial function on $\R^d$, compactly supported in the ball centered at the origin and with radius $1$, with $h(0)=1$ and such that for $G>(3d+1)/2$, and for some positive constant $C$,
\[
\left|\mathcal F_dh(t)\right|\le C\frac 1{(1+t)^{2G}}.
\]
Then
\begin{align*}
K_X(x,y)&:=\sum_{m=0}^Xh\left(\frac{\lambda_m}{\lambda_X}\right)\varphi_m(x)\overline{\varphi_m(y)}
\end{align*}
is an approximation to the identity, in the sense that for all $x\in \mathcal M$ and for all $\delta>0$,
\begin{align*}
&\int_{\M}K_X(x,y)dy=1,\\
&\int_{\M}|K_X(x,y)|dy=\int_{\M}|K_X(y,x)|dy\le C,\\
&\lim_{X\to+\infty}\int_{\{y:d(x,y)\ge\delta\}}|K_X(x,y)|dy=0.
\end{align*}

%\begin{itemize}
%\item[(i)] For all $x\in \mathcal M$, $\int_{\M}K_X(x,y)dy=1$.
%\item[(ii)] For all $x\in \mathcal M$, $\int_{\M}|K_X(x,y)|dy\le C$.
%\item[(iii)] For all $x\in \mathcal M$ and for all $\delta>0$, 
%\[
%\lim_{X\to+\infty}\int_{\{y:d(x,y)\ge\delta\}}|K_X(x,y)|dy=0.
%\]
%\end{itemize}
\end{corollary}
\begin{proof} It suffices to apply Theorem \ref{main} (i), with a sufficiently large $G$ to ensure the required integrability and decay.
\end{proof}
It follows by standard arguments that, for kernels as in Corollary \ref{cor}, the means 
\[
K_Xf(y):=\sum_{m=0}^{X}h\left(\frac{\lambda_m}{\lambda_X}\right)\widehat f(\lambda_m)\varphi_m(x)=\int_{\M}K_X(x,y)f(y)dy
\]
converge uniformly to $f(x)$ as $X\to+\infty$ whenever $f$ is continuous on $\M$, and in the $L^p$ norm whenever $f$ is in $L^p(\M)$, for $1\le p<+\infty$.

We have to mention here that Theorem \ref{main} is not entirely new and has been used already in several occasions by different authors (\cite{BC, BGG, CGT}). In \cite {BGG}, though, this result is not clearly stated, and its proof appears somehow mixed up with the result that the authors were actually proving, and for which they needed one such kernel.
In fact, essentially all the proofs of the theorems that we state here are already contained in \cite{BGG}. In \cite{CGT}, a vague statement is given, and for the proof the reader is referred to \cite{BC, T}. In \cite{BC} one can find a result, Theorem 2.3, which could be considered as one step of the proof of our result and that here corresponds somehow to our Theorem \ref{CB}. Our intent here is to give this result in the simplest and most transparent possible form, so that other authors can use it even if they do not master all the technicalities involved 
in the proof, like the Hadamard construction of the parametrix of the wave equation, see Section \ref{parametrix}. In Section \ref{proofs} we present all the steps needed to prove Theorem \ref{main}, and in the final Section \ref{application} we show how one can apply Theorem \ref{main} to give a direct proof of the main result of \cite{BGG}.

\section{Hadamard construction of the parametrix of the wave equation}\label{parametrix}

Let $\cos(\sqrt\Delta t)$ be the operator that associates to any function $f\in\mathcal D(\M)$
(smooth functions on $\M$), the solution $u\in\mathcal D'(\M)$ (distributions on $\M$) to the wave problem
\[
\begin{cases}
(\partial_t^2+\Delta)u(t,x)=0 & (t,x)\in \R\times\M\\
u(0,x)=f(x) & x\in\M\\
\partial_t u(0,x)=0 & x\in\M.
\end{cases}
\] 
%It is known (see e.g. \cite[Theorem 2.4.2]{sogge}) that the above wave operator has unit propagation speed, at least within small geodetic balls.
It is easy to see that 
\[
(\cos(\sqrt\Delta t)f)(x)=\sum_{m=0}^{+\infty}\cos(t\lambda_m)\widehat f(\lambda_m)\varphi_m(x).
\]
Notice in particular that since $\widehat f(\lambda_m)$ decays rapidly and $\|\varphi_m\|_\infty$ has polynomial growth, $(\cos(\sqrt\Delta t)f)(x)$ is in fact a smooth function, and as a distribution acts on smooth functions by integration
\begin{align*}
\langle \cos(\sqrt\Delta t) f, g\rangle_{\mathcal D'(\M)}&=\int_{\M}\left(\sum_{m=0}^{+\infty}\cos(t\lambda_m)\widehat f(\lambda_m)\varphi_m(x)\right)g(x)dx\\
&=\sum_{m=0}^{+\infty}\cos(t\lambda_m)\widehat f(\lambda_m)\int_{\M}\varphi_m(x)g(x)dx.
\end{align*}
%By the Schwartz kernel theorem, for every $t$ there exists a unique kernel $K_t\in\mathcal D'(\M\times\M)$ such that for every $f,g\in\mathcal D(\M)$,
%\[
%\langle \cos(\sqrt\Delta t) f, g\rangle_{\mathcal D'(\M)}=\langle K_t, f\otimes g\rangle_{\mathcal D'(\M\times\M)}.
%\]
Observe now that for every $f,g\in\mathcal D(\M)$, the function $t\mapsto \langle \cos(\sqrt\Delta t) f, g\rangle_{\mathcal D'(\M)}$ is bounded and continuous, and this implies that it can be seen as a tempered distribution in $\mathcal S'(\R)$. It obviously acts on smooth, rapidly decaying functions $h\in\mathcal S(\R)$ by integration
\begin{align*}
&\left\langle\langle \cos(\sqrt\Delta t) f, g\rangle_{\mathcal D'(\M)},h\right\rangle_{\mathcal S'(\R)}=\int_{\R} \langle \cos(\sqrt\Delta t) f, g\rangle_{\mathcal D'(\M)}h(t)dt.\\
%&=\sum_{m=0}^{+\infty}\widehat f(\lambda_m)\int_{\M}g(x){\varphi_m(x)}dx\int_{\R} h(t)\cos(\lambda_mt)dt\\
%&=\sum_{m=0}^{+\infty}\widehat f(\lambda_m)\overline{\widehat{\overline g}}(\lambda_m)\mathcal C_1h(\lambda_m),
\end{align*}
In particular, notice that by the above formula, $\cos(\sqrt{\Delta}\cdot)$ can be seen as a (tempered) distribution on $\mathcal M\times\mathcal M\times\R$.

%Let us now take an even continuous function $H\in L^1(\R)$, and assume that its cosine transform $\mathcal C H(s)=\int_{\R} H(t)\cos(st)dt$ is compactly supported. Then
%%, by \cite[Corollary 1.26]{SW}, 
%for every $s\in\R$,
%\[
%H(s)=\int_{\R}\mathcal C^{-1}H(t)\cos(st)dt
%\]
%where $\mathcal C^{-1} H(t)=\frac1{2\pi}\int_{\R} H(s)\cos(st)ds=\frac 1{2\pi}\mathcal C H(t)$ is the inverse cosine transform.
%
%Consider the operator $\mathcal H$ that maps any function $f\in\mathcal D(\M)$ to the distribution
%$
%\mathcal Hf=\sum_{m=0}^{+\infty}H(\lambda_m)\widehat f(\lambda_m)\varphi_m.
%$
%Since $\mathcal H f$ is in fact a smooth function,
%it acts on $\mathcal D(\M)$ by integration,
%\begin{align}\label{pairing}
%&\langle \mathcal H f,g\rangle_{\mathcal D'(\M)}
%=\int_{\M} \left(\sum_{m=0}^{+\infty}H(\lambda_m)\widehat f(\lambda_m)\varphi_m(x)\right)g(x)dx\nonumber\\
%&=\sum_{m=0}^{+\infty}H(\lambda_m)\widehat f(\lambda_m)\int_{\M}g(x) \varphi_m(x)dx\nonumber\\
%&=\sum_{m=0}^{+\infty}\left(\int_{\R}\mathcal C^{-1}H(t)\cos(\lambda_mt)dt\right)\widehat f(\lambda_m)\int_{\M}g(x) \varphi_m(x)dx\nonumber\\
%&=\int_{\R}\mathcal C^{-1}H(t)\sum_{m=0}^{+\infty}\cos(\lambda_mt)\widehat f(\lambda_m)\left(\int_{\M}g(x) \varphi_m(x)dx\right)dt\nonumber\\
%&=\int_{\R}\mathcal C^{-1}H(t)\langle \cos(\sqrt\Delta t) f, g\rangle_{\mathcal D'(\M)}dt
%\end{align}

The following asymptotic expansion of the solution 
of the above mentioned wave problem is due to Hadamard, and its principal term is known as Hadamard parametrix.

\begin{theorem}
[{see {\cite[Theorem 3.1.5]{sogge}}}]\label{sogge}Given a $d$-dimensional
Riemannian manifold $\left(  \mathcal{M},g\right)  $, there exists
$\varepsilon>0$ and functions $\alpha_{\nu}\in\mathcal{C}^{\infty}%
(\mathcal{M}\times\mathcal{M})$, so that if $Q>d+3$ then for every $f\in\mathcal  D(\M)$
\begin{align}\label{parametrix}
\left(\cos(  t\sqrt{\Delta})f\right)(x)&=\int_{\M}
\sum_{\nu=0}^{Q}\alpha_{\nu}(  x,y)
\partial_{t}(  E_{\nu}-\check{E}_{\nu})  (  t,d(
x,y)  )f(y)dy%\\
%\nonumber
%&
+\int_{\M}R_{Q}(  t,x,y)f(y)dy
\end{align}%
where $R_{Q}\in\mathcal{C}^{Q-d-3}(  [  -\varepsilon,\varepsilon
]  \times\mathcal{M}\times\mathcal{M})  $ and%
\[
\left\vert \partial_{t,x,y}^{\beta}R_{Q}(  t,x,y)  \right\vert \leq
C\left\vert t\right\vert ^{2Q+2-d-\left\vert \beta\right\vert }.
\]
Furthermore $\alpha_{0}(  x,y)  \ge0$ in $\mathcal{M}\times
\mathcal{M}$, and $\alpha_0(x,x)=1$.
\end{theorem}

Here we only want to recall that $E_\nu$ is a homogeneous distribution of degree $2\nu-d+1$ supported on the forward light cone $\{(t,x)\in\mathbb R^{1+d}:t\ge0, t^2\ge|x|^2\}$, radial in $x$, and defined by
\[
E_\nu(t,x)=\lim_{\varepsilon\to0+}\nu!(2\pi)^{-d-1}
\iint_{\mathbb R^{1+d}}e^{i(x\cdot\xi+tr)}(|\xi|^2-(\tau-i\varepsilon)^2)^{-\nu-1}d\xi d\tau.
\]
The distribution $\check E_\nu$ is the reflection of $E_\nu$
about the origin of $\mathbb R^{1+d}$. The distribution $E_0$ is the fundamental solution of the wave operator supported on the forward light cone, and for all $\nu=1,2,\ldots,$ the distributions $E_\nu$ are defined in such a way that $(\partial_t^2+\Delta)E_\nu=\nu E_{\nu-1}$. With a small abuse of notation, we shall sometimes write 
$\partial_{t}(  E_{\nu}-\check{E}_{\nu})  (  t,|z|)$
instead of $\partial_{t}(  E_{\nu}-\check{E}_{\nu})  (  t,z)$.
Formula \eqref{parametrix} has then to be interpreted in local coordinates (more precisely, normal coordinates in a  coordinate patch centered at $x\in\mathcal M$), whenever the time $t$ is smaller than the injectivity radius.

Finally, the distributions $\partial_{t}(  E_{\nu}-\check{E}_{\nu})  (  t,z)$ can be regarded as continuous radial functions of $z$ with values in $\mathcal S'(\mathbb R)$. Furthermore, when $0\le\nu<d/2$,
for every $z\in\mathbb{R}^{d}$  the inverse cosine transform 
$\mathcal{C}^{-1}\left(  \partial_{\cdot}(  E_{\nu}-\check{E}_{\nu
}) (  \cdot,z)  \right)  $ is a function and for all
$t\in\mathbb{R}$%
\begin{equation}
\mathcal{C}^{-1}\left(  \partial_{\cdot}(  E_{\nu}-\check{E}_{\nu
})  (  \cdot,z)  \right) (  t)  =\pi
^{-d/2}2^{-\nu-d/2-1}\vert t\vert ^{-2\nu-1+d}\frac{J_{-\nu
+d/2-1}(  t\vert z\vert )  }{(  t\vert
z\vert )  ^{-\nu+d/2-1}}, \label{identity}%
\end{equation}
whereas when $d/2\le\nu$, for every $z\in\mathbb{R}^{d}$ the distribution itself $\partial_{t}(  E_{\nu
}-\check{E}_{\nu}) (  t,z )  $ can be
identified with the locally integrable function%
\[
t\mapsto C_{\nu}|t|(  t^{2}-|z|  ^{2})  _{+}^{\nu-1+(
1-d)  /2},
\]
with $C_{\nu}=2^{-2\nu}\pi^{(1-d)/2}\left(1+\frac{1-d}{2} \right)$.

\section{Analysis of the kernel}\label{proofs}
Let us now take an even continuous function $H\in L^1(\R)$, and assume that its cosine transform $\mathcal C H(s)$ is compactly supported. Then
%, by \cite[Corollary 1.26]{SW}, 
for every $s\in\R$,
\[
H(s)=\int_{\R}\mathcal C^{-1}H(t)\cos(st)dt.
\]

Consider the operator $\mathcal H$ that maps any function $f\in\mathcal D(\M)$ to the distribution
\[
\mathcal Hf=\sum_{m=0}^{+\infty}H(\lambda_m)\widehat f(\lambda_m)\varphi_m.
\]
Since $\mathcal H f$ is in fact a smooth function,
it acts on $\mathcal D(\M)$ by integration,
\begin{align}\label{pairing}
&\langle \mathcal H f,g\rangle_{\mathcal D'(\M)}
=\int_{\M} \left(\sum_{m=0}^{+\infty}H(\lambda_m)\widehat f(\lambda_m)\varphi_m(x)\right)g(x)dx\nonumber\\
&=\sum_{m=0}^{+\infty}H(\lambda_m)\widehat f(\lambda_m)\int_{\M}g(x) \varphi_m(x)dx%\nonumber
=\sum_{m=0}^{+\infty}\left(\int_{\R}\mathcal C^{-1}H(t)\cos(\lambda_mt)dt\right)\widehat f(\lambda_m)\int_{\M}g(x) \varphi_m(x)dx\nonumber\\
&=\int_{\R}\mathcal C^{-1}H(t)\sum_{m=0}^{+\infty}\cos(\lambda_mt)\widehat f(\lambda_m)\left(\int_{\M}g(x) \varphi_m(x)dx\right)dt%\nonumber
=\int_{\R}\mathcal C^{-1}H(t)\langle \cos(\sqrt\Delta t) f, g\rangle_{\mathcal D'(\M)}dt.
\end{align}

\begin{theorem}\label{teo3}
Let $H\in L^1(\R)$ be even and continuous, and assume that its cosine transform $\mathcal C H(s)=\int_{\R} H(t)\cos(st)dt$ is supported in $[-\varepsilon,\varepsilon]$.
Let $\mathcal H$ be the operator that maps any function $f\in\mathcal D(\M)$ to the distribution
$
\mathcal Hf=\sum_{m=0}^{+\infty}H(\lambda_m)\widehat f(\lambda_m)\varphi_m.
$
Then, for any $f,g\in\mathcal D(\mathcal M)$,
\begin{align*}
&\langle \mathcal H f,g\rangle_{\mathcal D'(\M)}\\
&=\sum_{0\le\nu<d/2}
\int_{\mathcal M}\int_{\mathcal M}\alpha_\nu(x,y)\int_{0}^{+\infty}H(t) \pi
^{-d/2}2^{-\nu-d/2} t ^{-2\nu-1+d}\frac{J_{-\nu
+d/2-1}\left(  t d(x, y) \right)  }{\left(  td(x,y) \right)  ^{-\nu+d/2-1}}dt g(x)dxf(y)dy
\\
&+\sum_{d/2\le \nu\le Q}
C_{\nu}\int_{\mathcal M}\int_{\mathcal M}\alpha_\nu(x,y)\int_{-\varepsilon}^{\varepsilon}\mathcal C^{-1}H(t)
|t|\left(  t^{2}-d(x,y) ^{2}\right)  _{+}^{\nu-1+\left(
1-d\right)  /2}g(x)dx f(y)dy\\
&+
\int_{\mathcal M} \int_{\mathcal M}\int_{-\varepsilon}^{\varepsilon}\mathcal C^{-1}H(t)R_Q(t,x,y)dtg(x) f(y)dxdy.
\end{align*}
\end{theorem}

\begin{proof} 
It follows from \eqref{pairing} and Theorem \ref{sogge}, that
if $\mathcal C^{-1}H$ is supported in $[-\varepsilon,\varepsilon]$ then
\begin{align*}
&\langle \mathcal H f,g\rangle_{\mathcal D'(\M)}
=\int_{-\varepsilon}^{\varepsilon}\mathcal C^{-1}H(t)\langle \cos(\sqrt\Delta t) f, g\rangle_{\mathcal D'(\M)}dt\\
&=\sum_{\nu=0}^Q\int_{-\varepsilon}^{\varepsilon}\mathcal C^{-1}H(t)\langle \int_{\mathcal M}\alpha_\nu(\cdot,y)\partial_t(E_\nu-\check E_\nu)(t,d(\cdot,y)) f(y)dy, g\rangle_{\mathcal D'(\M)}dt\\
&+\int_{-\varepsilon}^{\varepsilon}\mathcal C^{-1}H(t)\langle \int_{\mathcal M}R_Q(t,\cdot,y) f(y)dy, g\rangle_{\mathcal D'(\M)}dt\\
&=\sum_{\nu=0}^Q\int_{-\varepsilon}^{\varepsilon}\mathcal C^{-1}H(t)\int_{\mathcal M}\langle \alpha_\nu(\cdot,y)\partial_t(E_\nu-\check E_\nu)(t,d(\cdot,y)),g\rangle_{\mathcal D'(\M)} f(y)dydt\\
&+\int_{-\varepsilon}^{\varepsilon}\mathcal C^{-1}H(t)\int_{\mathcal M} \int_{\mathcal M}R_Q(t,x,y)g(x)dx f(y)dy dt\\
&=\sum_{\nu=0}^Q\int_{\mathcal M}\int_{-\varepsilon}^{\varepsilon}\mathcal C^{-1}H(t)\langle 
\alpha_\nu(\cdot,y)\partial_t(E_\nu-\check E_\nu)(t,d(\cdot,y)),g\rangle_{\mathcal D'(\M)} f(y)dtdy\\
&+\int_{\mathcal M} \int_{\mathcal M}\int_{-\varepsilon}^{\varepsilon}\mathcal C^{-1}H(t)R_Q(t,x,y)dtg(x) f(y)dxdy \\
&=\sum_{\nu=0}^Q\int_{\mathcal M}\langle \alpha_\nu(\cdot,y)\int_{-\varepsilon}^{\varepsilon}\mathcal C^{-1}H(t)
\partial_t(E_\nu-\check E_\nu)(t,d(\cdot,y))dt,g\rangle_{\mathcal D'(\M)} f(y)dy\\
&+\int_{\mathcal M} \int_{\mathcal M}\int_{-\varepsilon}^{\varepsilon}\mathcal C^{-1}H(t)R_Q(t,x,y)dtg(x) f(y)dxdy.
\end{align*}

Let us now look closely to each of the terms of the above sum. If $0\le\nu<d/2$, then

\begin{align*}
&\int_{\mathcal M}\left\langle \alpha_\nu(\cdot,y)\int_{-\varepsilon}^{\varepsilon}\mathcal C^{-1}H(t)
\partial_t(E_\nu-\check E_\nu)(t,d(\cdot,y))dt,g\right\rangle_{\mathcal D'(\M)} f(y)dy\\
&=\int_{\mathcal M}\left\langle \alpha_\nu(\cdot,y)\int_{-\infty}^{+\infty}H(t) \mathcal C^{-1}
\left(\partial_t(E_\nu-\check E_\nu)(t,d(\cdot,y))\right)dt,g\right\rangle_{\mathcal D'(\M)} f(y)dy\\
&=\int_{\mathcal M}\left\langle \alpha_\nu(\cdot,y)\int_{-\infty}^{+\infty}H(t) \pi
^{-d/2}2^{-\nu-d/2-1}\left\vert t\right\vert ^{-2\nu-1+d}\frac{J_{-\nu
+d/2-1}\left(  t d(\cdot, y) \right)  }{\left(  t d(\cdot, y) \right)  ^{-\nu+d/2-1}}dt,g\right\rangle_{\mathcal D'(\M)} f(y)dy\\
\end{align*}
and since now for any $y$, the distribution acting on $g$ is a locally integrable function, it acts by integration, thus obtaining
\[
\int_{\mathcal M}\int_{\mathcal M}\alpha_\nu(x,y)\int_{0}^{+\infty}H(t) \pi
^{-d/2}2^{-\nu-d/2} t ^{-2\nu-1+d}\frac{J_{-\nu
+d/2-1}\left(  t d(x, y) \right)  }{\left(  td(x,y) \right)  ^{-\nu+d/2-1}}dt g(x)dxf(y)dy.
\]
If instead $\nu\ge d/2$, then 
\begin{align*}
&\int_{\mathcal M}\langle \alpha_\nu(\cdot,y)\int_{-\varepsilon}^{\varepsilon}\mathcal C^{-1}H(t)
\partial_t(E_\nu-\check E_\nu)(t,d(\cdot,y))dt,g\rangle_{\mathcal D'(\M)} f(y)dy\\
&
=C_{\nu}\int_{\mathcal M}\langle \alpha_\nu(\cdot,y)\int_{-\varepsilon}^{\varepsilon}\mathcal C^{-1}H(t)
|t|\left(  t^{2}-d(\cdot,y) ^{2}\right)  _{+}^{\nu-1+\left(
1-d\right)  /2},g\rangle_{\mathcal D'(\M)} f(y)dy.
\end{align*}
Again, for any $y$, the distribution acting on $g$ is a locally integrable function, so that the above term equals
\[
C_{\nu}\int_{\mathcal M}\int_{\mathcal M}\alpha_\nu(x,y)\int_{-\varepsilon}^{\varepsilon}\mathcal C^{-1}H(t)
|t|\left(  t^{2}-d(x,y) ^{2}\right)  _{+}^{\nu-1+\left(
1-d\right)  /2}g(x)dx f(y)dy.
\]
\end{proof}

The formula in Theorem \ref{teo3} also gives an explicit expression of the Schwartz kernel of $\mathcal H$. In particular, it is a function.

\begin{theorem}
\label{CB}
Let $H\in L^1(\R)$ be even and continuous, and assume that its cosine transform $\mathcal C H(s)=\int_{\R} H(t)\cos(st)dt$ is supported in $[-\varepsilon,\varepsilon]$.
Then
\begin{align*}
\sum_{m=0}^{+\infty}H(\lambda_m)\varphi_m(x)\overline{\varphi_m(y)}
&=\sum_{0\le\nu<d/2}
\alpha_\nu(x,y)\int_{0}^{+\infty}H(t) \pi
^{-d/2}2^{-\nu-d/2} t ^{-2\nu-1+d}\frac{J_{-\nu
+d/2-1}\left(  t d(x, y) \right)  }{\left(  td(x,y) \right)  ^{-\nu+d/2-1}}dt 
\\
&+\sum_{d/2\le \nu\le Q}
C_{\nu}\alpha_\nu(x,y)\int_{-\varepsilon}^{\varepsilon}\mathcal C^{-1}H(t)
|t|\left(  t^{2}-d(x,y) ^{2}\right)  _{+}^{\nu-1+\left(
1-d\right)  /2}dt\\
&+
\int_{-\varepsilon}^{\varepsilon}\mathcal C^{-1}H(t)R_Q(t,x,y)dt.
\end{align*}
\end{theorem}

\begin{proof} Since 
\[
\langle \mathcal H f,g\rangle_{\mathcal D'(\M)}
=\int_{\M}\int_{\M}\sum_{m=0}^{+\infty}H(\lambda_m)\varphi_m(x)\overline{\varphi_m(y)}g(x) f(y)dxdy,
\]
it follows that the kernel can also be written as the function
\[
\sum_{m=0}^{+\infty}H(\lambda_m)\varphi_m(x)\overline{\varphi_m(y)}.
\]
By Theorem \ref{teo3}, one has the thesis.
\end{proof}

For smooth radial integrable functions on $\mathbb{R}^{d}$, $f\left(  x\right)  =f_{0}\left(  \left\vert x\right\vert
\right) $, the Fourier transform 
\begin{align*}
\mathcal{F}_{d}f\left(  \xi\right)   &  =\int_{\mathbb{R}^{d}}f\left(
x\right)  e^{-2\pi ix\cdot\xi}dx
\end{align*}
reduces essentially to the Hankel
transform, given by (see \cite[Chapter 4, Theorem 3.3]{SW})
\begin{align}
\mathcal{F}_{d}f(\xi)   &  =2\pi\vert \xi\vert
^{-\frac{d-2}{2}}\int_{0}^{\infty}f_{0}(s)  J_{\frac{d-2}{2}}(  2\pi\vert \xi\vert s)  s^{\frac{d}{2}}ds.\label{Hankel}
\end{align}
As we mentioned before, with an abuse of notation, we will identify the function $f$
with its radial profile $f_{0}$ and write $\mathcal{F}_{d}f(\vert\xi\vert)  $ instead of $\mathcal{F}_{d}f(\xi).$
One can easily show that if $f$ is an even smooth function on $\R$, then for any $t\in\R$,
\[
\mathcal{C}^{-1}f(t)=\frac 1{2\pi}\mathcal{C}f(t)=\frac1{2\pi}\mathcal{F}_{1}f\left( \frac {|t|}{2\pi}\right).
\]
With this notation, our kernel can be rewritten as
\begin{align}\label{formula}
\sum_{m=0}^{+\infty}H(\lambda_m)\varphi_m(x)\overline{\varphi_m(y)}
&=\sum_{0\le\nu<d/2}
\alpha_\nu(x,y)\frac{\pi^\nu}{(2\pi)^d}\mathcal F_{d-2\nu}H\left(\frac{d(x,y)}{2\pi}\right)\nonumber\\
&+\frac 1\pi\sum_{d/2\le \nu\le Q}
C_{\nu}\alpha_\nu(x,y)\int_{d(x,y)}^{+\infty}\mathcal F_1H\left(\frac t{2\pi}\right)
t\left(  t^{2}-d(x,y) ^{2}\right) ^{\nu-1+\left(
1-d\right)  /2}dt\nonumber\\
&+
\frac 1\pi\int_{0}^{+\infty}\mathcal F_1H\left(\frac t{2\pi}\right)R_Q(t,x,y)dt.
\end{align}
We have the following

\begin{theorem}\label{teo5}
Let $H\in L^1(\R)$ be even and continuous, and assume that its Fourier transform $\mathcal F_d H$ is supported in $[0,\varepsilon/(2\pi))$.
Then
\begin{align*}
&\left|\sum_{m=0}^{+\infty}H(\lambda_m)\varphi_m(x)\overline{\varphi_m(y)}-\alpha_0(x,y)\dfrac{1}{(2\pi)^d}\mathcal F_dH\left(\frac{d(x,y)}{2\pi}\right)\right|\\
&\le C\sum_{1\le\nu\le Q}
\int_{d(x,y)}^{+\infty}r^{2\nu-1}\left|\mathcal F_dH\left(\frac r{2\pi}\right)\right|dr+C
\int_{0}^{+\infty}r^{2Q+1}\left|\mathcal F_dH\left(\frac r{2\pi}\right)\right|dr.
\end{align*}
\end{theorem}

\begin{proof}
We want to express the formula \eqref{formula} in terms of $\mathcal F_dH$ rather then $\mathcal F_1H$ or $\mathcal F_{d-2\nu}H$. This can be done by means of the following transplantation result (see \cite[eq. (3.9)]{st}):  for $d>d'\ge1$,%
\begin{equation}
\mathcal{F}_{d'}H (s)
=c_{d,d'}\int_{s}^{+\infty}(  r^{2}-s^{2})  ^{\left(
d-d'\right)  /2-1}r\mathcal F_dH(  r)  dr. \label{Transplantation}%
\end{equation}

Thus, if $\mathcal F_dH$ is supported in $[0,\varepsilon/(2\pi)]$, then $\mathcal F_1H$ is supported in $[0,\varepsilon/(2\pi)]$ too, and $\mathcal C H$ is supported in $[-\varepsilon,\varepsilon]$, as required. Also, if $\mathcal F_d H$ is nonnegative, so is $\mathcal F_1 H$.  

Let us now assume $1\le\nu<d/2$. Then
\begin{align*}
\left|\mathcal F_{d-2\nu}H\left(\dfrac{d(x,y)}{2\pi}\right)\right|
&=c_{d,d-2\nu}\left|\int_{d(x,y)/(2\pi)}^{+\infty}\left(  r^{2}-\dfrac{d(x,y)^{2}}{(2\pi)^2}\right)  ^{
\nu-1}r\mathcal F_dH(  r)  dr\right|\\
&=\dfrac{c_{d,d-2\nu}}{(2\pi)^{2\nu}}\left|\int_{d(x,y)}^{+\infty}\left(  r^{2}-{d(x,y)^{2}}\right)  ^{
\nu-1}r\mathcal F_dH\left(  \dfrac{r}{2\pi}\right)  dr\right|\\
&\le C \int_{d(x,y)}^{+\infty} r^{2\nu-1}\left|\mathcal F_dH\left(  \dfrac{r}{2\pi}\right)\right|  dr.
\end{align*}

Similarly, for $d\ge2$
\begin{align*}
&\int_{d(x,y)}^{+\infty}\mathcal F_1H\left(\dfrac t{2\pi}\right)
t(  t^{2}-d(x,y) ^{2}) ^{\nu-1+(
1-d)  /2}dt\\
&=c_{d}\int_{d(x,y)}^{+\infty}
\int_{t/(2\pi)}^{+\infty}\left(r^2-\dfrac{t^2}{(2\pi)^2}\right)^{(d-3)/2}r\mathcal F_dH(r)dr
t(  t^{2}-d(x,y) ^{2}) ^{\nu-1+(
1-d)  /2}dt\\
&=\dfrac{c_{d}}{(2\pi)^{d-1}}\int_{d(x,y)}^{+\infty}
r\mathcal F_dH\left(\dfrac r{2\pi}\right)
\int_{d(x,y)}^{r}(r^2-t^2)^{(d-3)/2}
t(  t^{2}-d(x,y) ^{2}) ^{\nu-1+(
1-d)  /2}dtdr.
\end{align*}
It can be proved easily that for some constant $\gamma$ depending on $d\ge 2$ and on $\nu$ between $d/2$ and $Q$, for all $r\ge d(x,y)$
\[
\int_{d(x,y)}^{r}(r^2-t^2)^{(d-3)/2}
t(  t^{2}-d(x,y) ^{2}) ^{\nu-1+(
1-d)  /2}dt\le\gamma (r^2-d(x,y)^2)^{\nu-1}.
\]
It follows that for all $d\ge 1$ and for all $d/2\le\nu\le Q$, 
\begin{align*}
&\left|\frac 1\pi\sum_{d/2\le \nu\le Q}
C_{\nu}\alpha_\nu(x,y)\int_{d(x,y)}^{+\infty}\mathcal F_1H\left(\dfrac t{2\pi}\right)
t(  t^{2}-d(x,y) ^{2}) ^{\nu-1+(
1-d)  /2}dt\right|\\
& \le C\sum_{d/2\le \nu\le Q} \int_{d(x,y)}^{+\infty}r \left|\mathcal F_d H\left(\dfrac r{2\pi}\right)\right|
(r^2-d(x,y)^2)^{\nu-1}dr \\
& \le C\sum_{d/2\le \nu\le Q} \int_{d(x,y)}^{+\infty}r^{2\nu-1} \left|\mathcal F_d H\left(\dfrac r{2\pi}\right)\right|
dr.
\end{align*}
The same strategy can be used to estimate the last term of the kernel, the one involving the remainder $R_Q$. Indeed,
\begin{align*}
\dfrac 1\pi\int_{0}^{+\infty}\mathcal F_1H\left(\dfrac t{2\pi}\right)R_Q(t,x,y)dt
&=\dfrac {c_d}\pi\int_{0}^{+\infty}\int_{t/2\pi}^{+\infty}
(r^2-(t/2\pi)^2)^{(d-3)/2}r\mathcal F_dH(r)dr
R_Q(t,x,y)dt\\
&=\dfrac {c_d}{\pi(2\pi)^{d-1}}\int_{0}^{+\infty} r\mathcal F_dH\left(\dfrac r
{2\pi}\right)\int_{0}^{r}
(r^2-t^2)^{(d-3)/2}
R_Q(t,x,y)dtdr.
\end{align*}
It follows that 
\[
\left|\frac 1\pi\int_{0}^{+\infty}\mathcal F_1H\left(\dfrac t{2\pi}\right)R_Q(t,x,y)dt\right|
\le C\int_{0}^{+\infty} \left|\mathcal F_dH\left(\dfrac r
{2\pi}\right)\right| r^{2Q+1}dr.
\]\end{proof}

Let us now fix a more specific choice for the function $H$. 
Let $h$ be an integrable radial function on $\R^d$ and let $\eta$ be a continuous integrable radial function on
$\R^d$ with Fourier transform compactly supported in the ball centered at the origin and with radius $\varepsilon/(2\pi)$.
Let us fix a nonnegative integer $X$, and define
\[
H(|z|)=h\left(\frac{\cdot}{\lambda_X}\right)\ast\eta(z),
\]
where the convolution is intended in $\R^d$. Observe that with the above choices, $H$ is continuous, it belongs to $L^1(\R)$ and 
\[
\mathcal F_d H(t)=\lambda_X^d\mathcal F_d h(\lambda_X t)\mathcal F_d\eta(t)
\] 
is supported in $[0,\varepsilon/(2\pi)]$, so that the previous theorem can be applied. 
%We will also assume that for some $G> Q+1$, and for some positive constant $C$,
%\[
%\left|\mathcal F_dh(t)\right|\le C\frac 1{(1+t)^{2G}}.
%\]

\begin{theorem}
Let $h$ be an integrable radial function on $\R^d$ such that for some $G> 0$, and for some positive constant $C$,
\[
\left|\mathcal F_dh(t)\right|\le C\dfrac 1{(1+t)^{2G}}.
\]
Let $\eta$ be a continuous integrable radial function on
$\R^d$ with Fourier transform $\mathcal F_d\eta$ compactly supported in the ball centered at the origin and with radius $\varepsilon/(2\pi)$.
Let us fix a nonnegative integer $X$, and define
\[
H(|z|)=h\left(\dfrac{\cdot}{\lambda_X}\right)\ast\eta(z),
\]
where the convolution is intended in $\R^d$.
Then
\begin{align*}
&\left|\sum_{m=0}^{+\infty}H(\lambda_m)\varphi_m(x)\overline{\varphi_m(y)}-\alpha_0(x,y)\dfrac{1}{(2\pi)^d}\mathcal F_dH\left(\dfrac{d(x,y)}{2\pi}\right)\right|\\
&\qquad\qquad\qquad\qquad\qquad\qquad\le C
\begin{cases}
\lambda_X^{d-2G} & \text{ if } 0<G<1,\\
\\
\dfrac{\lambda_X^{d-2}}{(1+\lambda_X d(x,y))^{2G-2}} & \text{ if } G>1,\, G \text{ non integer,}\\
\\
\dfrac{\lambda_X^{d-2}}{(1+\lambda_X d(x,y))^{2G-2}}+ \lambda_X^{d-2G}\log{\lambda_X} & \text{ if } G\ge1,\, G \text{ integer.}
\end{cases}
\end{align*}
\end{theorem}

\begin{proof} 
Let $Q$ be an integer greater than $d+1$ and than $G-1$. Then we may apply Theorem \ref{teo5}. For all $1\le\nu\le Q$,
\begin{align*}
\int_{d(x,y)}^{+\infty}r^{2\nu-1}\left|\mathcal F_dH\left(\frac r{2\pi}\right)\right|dr
&\le C \lambda_X^d\|\mathcal F_d \eta\|_\infty\int_{d(x,y)}^{\varepsilon}
\frac{1}{(1+\lambda_X r/(2\pi))^{2G}}r^{2\nu-1}dr\\
&= C (2\pi)^{2\nu} \lambda_X^{d-2\nu}\|\mathcal F_d \eta\|_\infty\int_{\lambda_X d(x,y)/(2\pi)}^{\lambda_X\varepsilon/2\pi}
\frac{r^{2\nu-1}}{(1+r)^{2G}}dr\\
&\le C \lambda_X^{d-2\nu}\int_{\lambda_X d(x,y)/(2\pi)}^{\lambda_X\varepsilon/2\pi}
\frac{1}{(1+r)^{2G-2\nu+1}}dr\\
&\le C
\begin{cases}
 \dfrac{\lambda_X^{d-2\nu}}{(1+\lambda_X d(x,y))^{2G-2\nu}} & \text{ if } \nu<G,\\
 \\
\lambda_X^{d-2G}& \text{ if } \nu>G, \\
\\
\lambda_X^{d-2G}\log(\lambda_X)   & \text{ if } \nu=G \text{ (with $G$ integer).}
\end{cases}
\end{align*}
On the other hand,
\begin{align*}
\int_{0}^{+\infty}r^{2Q+1}\left|\mathcal F_dH\left(\frac r{2\pi}\right)\right|dr
&\le C \lambda_X^d\|\mathcal F_d \eta\|_\infty\int_{0}^{\varepsilon}
\frac{1}{(1+\lambda_X r/(2\pi))^{2G}}r^{2Q+1}dr\\
&= C (2\pi)^{2Q+2} \lambda_X^{d-2Q-2}\|\mathcal F_d \eta\|_\infty\int_{0}^{\lambda_X\varepsilon/2\pi}
\frac{r^{2Q+1}}{(1+r)^{2G}}dr\\
&\le C \lambda_X^{d-2G}.
\end{align*}
It now follows from Theorem \ref{teo5} that
\begin{align*}
&\left|\sum_{m=0}^{+\infty}H(\lambda_m)\varphi_m(x)\overline{\varphi_m(y)}-\alpha_0(x,y)\dfrac{1}{(2\pi)^d}\mathcal F_dH\left(\frac{d(x,y)}{2\pi}\right)\right|\\
&\le C\sum_{1\le\nu<G}
\dfrac{\lambda_X^{d-2\nu}}{(1+\lambda_X d(x,y))^{2G-2\nu}}
+C \sum_{\nu=G}\lambda_X^{d-2G}\log(\lambda_X)
+C\sum_{G<\nu\le Q}\lambda_X^{d-2G},
\end{align*}
and the thesis follows.
\end{proof}

We are now ready for the final step. The kernel we have found so far is not what we wanted, since $H$ is not supported in $[0,\lambda_X]$. Therefore we need some further assumptions on $h$ and $\eta$. 
%More precisely, let us as usual assume that  $h$ is an integrable radial function on $\R^d$ such that for some $G> (d+2)/2$, and for some positive constant $C$,
%\[
%\left|\mathcal F_dh(t)\right|\le C\frac 1{(1+t)^{2G}},
%\]
%but let us also assume that $h$ is compactly supported in the ball centered at the origin and with radius $1$.
%By \cite[Theorem 1.7]{SW}, the decay of the Fourier transform of $h$  implies that $h\in\mathcal C^{\lfloor 2G\rfloor-d-1}(\R^d)$.
%Concerning $\eta$, let us simply assume that it is the Fourier transform
%of a radial function $\phi\in\mathcal C^\infty(\R^d)$,
%which is compactly supported in the ball centered at the origin with radius $\varepsilon/(2\pi)$, and that equals $1$ in the ball centered at the origin with radius $\varepsilon/(4\pi)$. This of course accomplishes all the requests of the previous theorems.

\begin{theorem}
Let $h$ be an integrable radial function on $\R^d$ such that for some  $G> (d+2)/2$, and for some positive constant $C$,
\[
\left|\mathcal F_dh(t)\right|\le C\frac 1{(1+t)^{2G}},
\]
and assume that $h$ is compactly supported in the ball centered at the origin and with radius $1$. Let $\eta$ be a continuous integrable radial function on
$\R^d$ with Fourier transform $\mathcal F_d\eta$ compactly supported in the ball centered at the origin with radius $\varepsilon/(2\pi)$ and that equals $1$ in the ball centered at the origin with radius $\varepsilon/(4\pi)$.
Let $I(z)$ be defined by
\begin{equation*}
\label{functionI}
I(z)=h(z/\lambda_X)-H(z)=\int_{\mathbb{R}^{d}}\left[  h\left(  \frac{z}%
{\lambda_{X}}\right)  -h\left(  \frac{z-y}{\lambda_{X}}\right)  \right]
\eta(  y)  dy.
\end{equation*}
Then
\[
\left|\sum_{m=0}^{+\infty}I(\lambda_m)\varphi_m(x)\overline{\varphi_m(y)}\right|
\le c\lambda_X^{-\lfloor 2G\rfloor+3d}.
\]
\end{theorem}

\begin{proof}
Let us first give an estimate on the function
$I(z)$.
Since $\eta(  y)$ has rapid decay at infinity and $h(
z)  $ is supported in $\{|z|\le1\} $, if $\vert
z\vert \geq2\lambda_{X}$ we have%
\begin{align*}
\vert I(  z)  \vert  &  \leq\int_{\mathbb{R}^{d}%
}\left\vert h\left(  \frac{z-y}{\lambda_{X}}\right)  \eta(  y)
\right\vert dy\leq\int_{\left\{  \vert z-y\vert \leq\lambda
_{X}\right\}  }\left\vert h\left(  \frac{z-y}{\lambda_{X}}\right)  \eta(
y)  \right\vert dy\\
&  \leq c\int_{\left\{  \vert y\vert \geq\vert z\vert
-\lambda_{X}\right\}  }\vert \eta(y)  \vert dy\leq
C(  1+\vert z\vert -\lambda_{X})  ^{-M},
\end{align*}
for some $M$ as large as needed.
Assume $\left\vert z\right\vert <2\lambda_{X}$. By \cite[Theorem 1.7]{SW}, the decay of the Fourier transform of $h$  implies that $h\in\mathcal C^{\lfloor 2G\rfloor-d-1}(\R^d)$. By Taylor's theorem with
integral reminder, setting $K=\lfloor 2G\rfloor-d-1\ge1$, we can write%
\begin{align*}
 h\left(  \frac{z}{\lambda_{X}}-\frac{y}{\lambda_{X}}\right)  
&=   h\left(  \frac{z}{\lambda_{X}}\right)  +\sum_{1\leq\left\vert
\alpha\right\vert \leq K-1}\frac{1}{\alpha!}\frac{\partial^{\left\vert
\alpha\right\vert }h}{\partial x^{\alpha}}\left(  \frac{z}{\lambda_{X}%
}\right)  \left(  -\frac{y}{\lambda_{X}}\right)  ^{\alpha}\\
&  +\sum_{\left\vert \alpha\right\vert =K}\frac{K}{\alpha!}\left(  -\frac
{y}{\lambda_{X}}\right)  ^{\alpha}\int_{0}^{1}(  1-t)  ^{K-1}%
\frac{\partial^{\left\vert \alpha\right\vert }h}{\partial x^{\alpha}}\left(
\frac{z}{\lambda_{X}}-t\frac{y}{\lambda_{X}}\right)  dt
\end{align*}
so that%
\[
h\left(  \frac{z}{\lambda_{X}}-\frac{y}{\lambda_{X}}\right)  =h\left(
\frac{z}{\lambda_{X}}\right)  +\sum_{1\leq\left\vert \alpha\right\vert \leq
K-1}\frac{1}{\alpha!}\frac{\partial^{\left\vert \alpha\right\vert }h}{\partial
x^{\alpha}}\left(  \frac{z}{\lambda_{X}}\right)  \left(  -\frac{y}{\lambda
_{X}}\right)  ^{\alpha}+O\left(  \left\vert -\frac{y}{\lambda_{X}}\right\vert
^{K}\right)  .
\]
It follows that%
\begin{align*}
I\left(  z\right)   &  =\int_{\mathbb{R}^{d}}\left[  h\left(  \frac{z}%
{\lambda_{X}}\right)  -h\left(  \frac{z-y}{\lambda_{X}}\right)  \right]
\eta\left(  y\right)  dy\\
&  =-\sum_{1\leq\left\vert \alpha\right\vert \leq K-1}\frac{1}{\alpha!}%
\frac{\partial^{\left\vert \alpha\right\vert }h}{\partial x^{\alpha}}\left(
\frac{z}{\lambda_{X}}\right)  \int_{\mathbb{R}^{d}}\left(  -\frac{y}%
{\lambda_{X}}\right)  ^{\alpha}\eta\left(  y\right)  dy+\int_{\mathbb{R}^{d}%
}O\left(  \frac{\left\vert y\right\vert ^{K}}{\lambda_{X}^{K}}\right)
\eta\left(  y\right)  dy
\end{align*}
and since
\[
\int_{\mathbb{R}^{d}}y^{\alpha}\eta(  y)  dy=\mathcal{F}_{d}(
y^{\alpha}\eta(  y)  ) (  0)  =(  -2\pi
i)  ^{-\vert \alpha\vert }\frac{\partial^{\vert
\alpha\vert }\mathcal{F}_{d}\eta}{\partial\xi^{\alpha}}(  0)
=0
\]
we obtain%
\[
\vert I(  z)  \vert \leq c\lambda_{X}^{-K}.
\]

The kernel
\[
\sum_{m=0}^{+\infty}I(\lambda_m)\varphi_m(x)\overline{\varphi_m(y)}
\]
can be estimated uniformly by means of Weyl's estimates on the eigenfunctions, $\|\varphi_m\|_\infty\le c(1+\lambda_m)^{(d-1)/2}$. Indeed, if $M>2d-1$,
\begin{align*}
\left|\sum_{m=0}^{+\infty}I(\lambda_m)\varphi_m(x)\overline{\varphi_m(y)}\right|
&\le c\lambda_X^{-K}\sum_{\lambda_m\le2\lambda_X}(1+\lambda_m)^{d-1}+c\sum_{\lambda_m\ge2\lambda_X}(1+\lambda_m-\lambda_X)^{-M}(1+\lambda_m)^{d-1}\\
&\le c\lambda_X^{-K+2d-1}+c\sum_{\lambda_m\ge2\lambda_X}
\lambda_m^{-M+d-1}\\
&\le c\lambda_X^{-K+2d-1}+c\sum_{k=1}^{+\infty}\sum_{2^k\lambda_X\le\lambda_m\le2^{k+1}\lambda_X}
\lambda_m^{-M+d-1}\\
&\le c\lambda_X^{-K+2d-1}+c\sum_{k=1}^{+\infty}
\lambda_X^d2^{dk}(\lambda_X2^k)^{-M+d-1}\\
&\le c\lambda_X^{-K+2d-1}+c\lambda_X^{-M+2d-1}\sum_{k=1}^{+\infty}
\left(2^{-M+2d-1}\right)^k\\
&\le c\lambda_X^{-K+2d-1}+c\lambda_X^{-M+2d-1}.
\end{align*}
Since we can take $M\ge K$, we have therefore proved the thesis. \end{proof}

We are ready to state our final result.
\begin{theorem}\label{preciso}
Let $h$ be an integrable radial function on $\R^d$ such that for some $G>(d+2)/2$, and for some positive constant $C$,
\[
\left|\mathcal F_dh(t)\right|\le C\frac 1{(1+t)^{2G}},
\]
and assume that $h$ is compactly supported in the ball centered at the origin and with radius $1$. Then
\begin{align*}
\left|\sum_{m=0}^{X}h\left(\frac{\lambda_m}{\lambda_X}\right)\varphi_m(x)\overline{\varphi_m(y)}-\alpha_0(x,y)\frac {\lambda_X^d}{(2\pi)^d}\mathcal F_dh\left(\frac{\lambda_Xd(x,y)}{2\pi}\right)\right|
&\le C
\frac{\lambda_X^{d-2}}{(1+\lambda_X d(x,y))^{2G-2}}
%+C \lambda_X^{d-2G}\log(\lambda_X)
+C\lambda_X^{3d-\lfloor 2G\rfloor}.
\end{align*}
\end{theorem}

\begin{proof} It suffices to observe that 
\begin{align*}
&\sum_{m=0}^{+\infty}h\left(\frac{\lambda_m}{\lambda_X}\right)\varphi_m(x)\overline{\varphi_m(y)}-\alpha_0(x,y)\frac {\lambda_X^d}{(2\pi)^d}\mathcal F_dh\left(\frac{\lambda_Xd(x,y)}{2\pi}\right)\\
=&\sum_{m=0}^{+\infty}I\left({\lambda_m}\right)\varphi_m(x)\overline{\varphi_m(y)}\\
+&\sum_{m=0}^{+\infty}H\left({\lambda_m}\right)\varphi_m(x)\overline{\varphi_m(y)}-\alpha_0(x,y)\frac 1{(2\pi)^d}\mathcal F_dH\left(\frac{d(x,y)}{2\pi}\right)\\
+&\alpha_0(x,y)\frac 1{(2\pi)^d}\mathcal F_dH\left(\frac{d(x,y)}{2\pi}\right)-\alpha_0(x,y)\frac {\lambda_X^d}{(2\pi)^d}\mathcal F_dh\left(\frac{\lambda_Xd(x,y)}{2\pi}\right).
\end{align*}
The estimates of the first two terms follow from the previous theorems. Concerning the last term, since
\[
\mathcal F_dH\left(\frac{d(x,y)}{2\pi}\right)=\lambda_X^d\mathcal F_dh\left(\frac{\lambda_Xd(x,y)}{2\pi}\right)\mathcal F_d\eta\left(\frac{d(x,y)}{2\pi}\right)
\]
and since $\mathcal F_d\eta(t)$ equals $1$ for $t\le\varepsilon/4\pi$, it follows that 
\[
\alpha_0(x,y)\frac 1{(2\pi)^d}\mathcal F_dH\left(\frac{d(x,y)}{2\pi}\right)-\alpha_0(x,y)\frac {\lambda_X^d}{(2\pi)^d}\mathcal F_dh\left(\frac{\lambda_Xd(x,y)}{2\pi}\right)
\]
equals zero when $d(x,y)\le \varepsilon/2$, and when $d(x,y)\ge\varepsilon/2$ it is bounded in absolute value by
\[
\|\alpha_0\|_\infty\frac 1{(2\pi)^d}\lambda_X^d\left(\|\mathcal F_d\eta\|_\infty+1\right)\frac C{(1+\lambda_Xd(x,y))^{2G}}
\le
C\lambda_X^{d-2G}.
\]\end{proof}

%Taking for example $G=Q+d+1$ gives the estimates
%\begin{align*}
%&\left|\sum_{m=0}^{X}h\left(\frac{\lambda_m}{\lambda_X}\right)\varphi_m(x)\overline{\varphi_m(y)}-\alpha_0(x,y)\frac {\lambda_X^d}{2\pi}\mathcal F_dh\left(\frac{\lambda_Xd(x,y)}{2\pi}\right)\right| \\
%&\le C\sum_{\nu=1}^Q
%\frac{\lambda_X^{d-2\nu}}{(1+\lambda_X d(x,y))^{2Q+2d+2-2\nu}}
%+C \lambda_X^{d-2Q-2}\\
%&\le C
%\frac{\lambda_X^{d-2}}{(1+\lambda_X d(x,y))^{2d+2}}
%+C \lambda_X^{d-2Q-2}.
%\end{align*}
%On the other hand, taking for example $G=2Q$,  gives the estimate
%\begin{align}
%\label{finale}
%\nonumber&\left|\sum_{m=0}^{X}h\left(\frac{\lambda_m}{\lambda_X}\right)\varphi_m(x)\overline{\varphi_m(y)}-\alpha_0(x,y)\frac {\lambda_X^d}{2\pi}\mathcal F_dh\left(\frac{\lambda_Xd(x,y)}{2\pi}\right)\right| \\
%\nonumber&\le C\sum_{\nu=1}^Q
%\frac{\lambda_X^{d-2\nu}}{(1+\lambda_X d(x,y))^{2G-2\nu}}
%+C \lambda_X^{d-2Q-2}+C \lambda_X^{-2G+3d}\\
%&\le C
%\frac{\lambda_X^{d-2}}{(1+\lambda_X d(x,y))^{2Q}}
%+C \lambda_X^{d-2Q-2}
%\le C
%\frac{\lambda_X^{d-2}}{(1+\lambda_X d(x,y))^{2Q}}
%\end{align}

\begin{proof}[Proof of Theorem \ref{main}.]
Since $G>d+1$, we can apply Theorem \ref{preciso}:
\begin{align*}
\left|\sum_{m=0}^{X}h\left(\frac{\lambda_m}{\lambda_X}\right)\varphi_m(x)\overline{\varphi_m(y)}-\alpha_0(x,y)\frac {\lambda_X^d}{(2\pi)^d}\mathcal F_dh\left(\frac{\lambda_Xd(x,y)}{2\pi}\right)\right|
%&\le C
%\frac{\lambda_X^{d-2}}{(1+\lambda_X d(x,y))^{2G-2}}
%+C \lambda_X^{d-2\lceil G\rceil+2}+C\lambda_X^{-\lfloor 2G\rfloor+3d}\\
&\le C
\frac{\lambda_X^{d-2}}{(1+\lambda_X d(x,y))^{2G-2}}
+C\lambda_X^{-\lfloor 2G\rfloor+3d}\\
&\le C
\frac{\lambda_X^{d-2}}{(1+\lambda_X d(x,y))^{\lfloor 2G\rfloor-2-2d}}.
\end{align*}
This proves point (i). As for point (ii),
it suffices to set 
$
h=\psi\ast\psi
$
where 
\[
\left|\mathcal F_d\psi(\xi)\right|\le C\frac 1{(1+|\xi|)^{G}},
\]
with $\psi$ a nonnegative radial function, compactly supported in the ball centered at the origin and with radius $1/2$ and with $\|\psi\|_2=1$.
\end{proof}

\section{An application}\label{application}

As we mentioned in the Introduction, an explicit expression of the kernel as 
a sum of a nonnegative term and a bounded remainder, allows to simplify the original proof of the  following theorem, a version of the Cassels--Montogomery inequality for compact manifolds recently proved in \cite{BGG},

\begin{theorem}
\label{cassels} There exists a positive constant $C$ such that for all integers
$N$ and $X$ and for all finite sequences of $N$ points in $\mathcal{M}$,
$\left\{  x\left(  j\right)  \right\}  _{j=1}^{N}$, and positive weights
$\left\{  a_{j}\right\}  _{j=1}^{N}$ we have%
\begin{equation}\label{cass}
\sum_{m=0}^{X}\left\vert \sum_{j=1}^{N}a_{j}\varphi_{m}\left(  x\left(
j\right)  \right)  \right\vert ^{2}\geq CX\sum_{j=1}^{N}a_{j}%
^{2}.
\end{equation}
\end{theorem}
\begin{proof}
It suffices to show the theorem for large $X$.
Let $Y=\kappa
X$, with $\kappa$ a positive integer which we will be chosen  later. By \cite[Theorem 2]{GL}, the manifold $\mathcal{M}$ can be split into $Y$
disjoint regions $\left\{  \mathcal{R}_{i}\right\}  _{i=1}^{Y}$ with measure
$\left\vert \mathcal{R}_{i}\right\vert =1/Y$ and such that each region
contains a ball of radius $c_{1}Y^{-1/d}$ and is contained in a ball of radius
$c_{2}Y^{-1/d}$, for appropriate values of $c_{1}$ and $c_{2}$ independent of
$Y$. Call $\left\{  \mathcal{B}_{r}\right\}  _{r=1}^{R}$ the sequence
of all the regions in $\left\{  \mathcal{R}_{i}\right\}
_{i=1}^{Y}$ which contain at least one of the points $x\left(  j\right)  $,
 $K_{r}$ the cardinality of the set $\left\{  j=1,\ldots,N:x\left(
j\right)  \in\mathcal{B}_{r}\right\}  $ and $S_{r}$ the sum of the weights
$\{a_{j}\}$ corresponding to points $x\left(  j\right)  \in\mathcal{B}_{r}$.
Assume without loss of generality that
\[
S_{1}\geq S_{2}\geq\ldots\geq S_{R}>0.
\]
Rename the sequence $\left\{  x\left(  j\right)  \right\}  _{j=1}^{N}$ as
\[
\left\{  x_{r,j}\right\}  _{\substack{r=1,\ldots,R\\j=1,\ldots,K_{r}}}
\]
with $x_{r,j}\in\mathcal{B}_{r}$ for all $j=1,\ldots,K_{r}$, and the sequence
$\left\{  a_{j}\right\}  _{j=1}^{N}$ as%
\[
\left\{  a_{r,j}\right\}  _{\substack{r=1,\ldots,R\\j=1,\ldots,K_{r}}}.
\]
Observe that
$
S_{r}=\sum_{j=1}^{K_{r}}a_{r,j}.
$
Inequality (\ref{cass}) follows immediately from%
\begin{equation}
\sum_{m=0}^{X}\left\vert \sum_{r=1}^{R}\sum_{j=1}^{K_{r}}a_{r,j}\varphi
_{m}\left(  x_{r,j}\right)  \right\vert ^{2}\geq CX\sum_{r=1}^{R}\left(\sum
_{j=1}^{K_{r}}a_{r,j}\right)^{2}.\label{lemma}%
\end{equation}
Notice that, if $h$ is as in the hypotheses of Theorem \ref{main}, then
\begin{align*}
  \sum_{m=0}^{X}\left\vert \sum_{r=1}^{R}\sum_{j=1}^{K_{r}}a_{r,j}\varphi
_{m}\left(  x_{r,j}\right)  \right\vert ^{2}
&  \geq\sum_{m=0}^{+\infty}h\left(  \frac{\lambda_{m}}{\lambda_{X}}\right)
\left\vert \sum_{r=1}^{R}\sum_{j=1}^{K_{r}}a_{r,j}\varphi_{m}\left(
x_{r,j}\right)  \right\vert ^{2}\label{somma}\\
&  =\sum_{r=1}^{R}\sum_{j=1}^{K_{r}}\sum_{s=1}^{R}\sum_{i=1}^{K_{s}}%
a_{r,j}a_{s,i}\left(  \sum_{m=0}^{+\infty}h\left(  \frac{\lambda_{m}}%
{\lambda_{X}}\right)  \varphi_{m}\left(  x_{r,j}\right)  \overline{\varphi
_{m}\left(  x_{s,i}\right)  }\right) \\
&\ge 
\sum_{r=1}^{R}\sum_{j=1}^{K_{r}}\sum_{s=1}^{R}\sum_{i=1}^{K_{s}}%
a_{r,j}a_{s,i}
\frac{\alpha_0(x_{r,j},x_{s,i})}{2\pi}\lambda_X^d {\mathcal F_d}h\left(\frac{\lambda_Xd(x_{r,j},x_{s,i})}{2\pi}\right)
\\
&\quad
-C\sum_{r=1}^{R}\sum_{j=1}^{K_{r}}\sum_{s=1}^{R}\sum_{i=1}^{K_{s}}%
a_{r,j}a_{s,i}
\frac{\lambda_X^{d-2}}{\left(1+\lambda_Xd(x_{r,j},x_{s,i})\right)^{\lfloor 2G\rfloor-2d-2}}.
\end{align*}
Let $\kappa$ large enough so that if $x,y\in\mathcal{B}_{r}$%
\[
\mathcal{F}_{d}h\left(  \frac{\lambda_{X}d\left(  x,y\right)  }{2\pi}\right)
 \geq\frac{ \mathcal{F}_{d}h  \left(
0\right)  }{2}>0.
\]
Thus%
\begin{align*}
&  \sum_{r=1}^{R}\sum_{j=1}^{K_{r}}\sum_{s=1}^{R}\sum_{i=1}^{K_{s}}%
a_{r,j}a_{s,i}\frac{\alpha_0(x_{r,j},x_{s,i})}{2\pi}\lambda_X^d {\mathcal F_d}h\left(\frac{\lambda_Xd(x_{r,j},x_{s,i})}{2\pi}\right) \\   &\geq CX\sum_{r=1}^{R}\sum_{j=1}^{K_{r}}\sum_{i=1}^{K_{r}}a_{r,j}%
a_{r,i}=CX\sum_{r=1}^{R}\left(  \sum_{j=1}^{K_{r}}a_{r,j}\right)  ^{2}.
\end{align*}

In order to estimate the remainder, let us call $z_r$ the center of the ball of radius
$c_{2}Y^{-1/d}$ containing the region $\mathcal{B}_{r}$ and let $c_{3}=10c_{2}$. For every $r=1,\ldots,R$ we
will consider separately the contribution of those values of $s$ for which
$\mathcal{B}_{s}$ is near $\mathcal{B}_{r}$ (meaning that $\mathcal{B}%
_{s}$ is contained in the ball centered at $z_{r}$ with radius
$c_{3}Y^{-1/d}$) and the contribution of the remaining values of $s$, for
which we will say that $\mathcal{B}_{s}$ is far from $\mathcal{B}_{r}$. Notice
that there are at most
\[
\frac{\left\vert B\left(  z_{r},c_{3}Y^{-1/d}\right)  \right\vert }{Y^{-1}%
}\leq\frac{C\left(  c_{3}Y^{-1/d}\right)  ^{d}}{Y^{-1}}\leq Cc_{3}^{d}%
\]
regions $\mathcal{B}_{s}$ near $\mathcal{B}_{r}$. Thus, since
$\lambda_{X}\sim X^{1/d}$ and since  $\sum
_{j=1}^{K_{r}}a_{r,j}\geq\sum_{i=1}^{K_{s}}a_{s,i}$ for $r\leqslant s$, setting $M=\lfloor 2G\rfloor -2d-2$ we obtain,
\begin{align*}
&  \sum_{r=1}^{R}\sum_{j=1}^{K_{r}}\sum_{s=r}^{R}\sum_{i=1}^{K_{s}}%
a_{r,j}a_{s,i}\frac{\lambda_{X}^{d-2}}{\left(  1+\lambda_{X}d\left(
x_{r,j},x_{s,i}\right)  \right)  ^{M}}\\
\leq &  CX^{1-2/d}\sum_{r=1}^{R}\sum_{\substack{s=r\,\\\mathcal{B}_{s}%
\,\text{near\thinspace}\mathcal{B}_{r}}}^{R}\sum_{j=1}^{K_{r}}a_{r,j}\sum_{i=1}^{K_{s}%
}a_{s,i}
  +CX^{1-2/d}\sum_{r=1}^{R}\sum_{\substack{s=r\,\\\mathcal{B}_{s}\text{\thinspace
far from\thinspace}\mathcal{B}_{r}}}^{R}\sum_{j=1}^{K_{r}}a_{r,j}\sum_{i=1}^{K_{s}%
}a_{s,i}\left(  \lambda_{X}d\left(  x_{r,j},x_{s,i}\right)  \right)  ^{-M}\\
\leq &  CX^{1-2/d}\sum_{r=1}^{R}\left(  \sum_{j=1}^{K_{r}}a_{r,j}\right)
^{2} +CX^{1-2/d}\sum_{r=1}^{R-1}\sum_{\substack{s=r+1\\\mathcal{B}_{s}\text{\thinspace
far from\thinspace}\mathcal{B}_{r}}}^{R}\sum_{j=1}^{K_{r}}a_{r,j}\sum_{i=1}^{K_{s}%
}a_{s,i}\left(  X^{1/d}d\left(  x_{r,j},x_{s,i}\right)  \right)  ^{-M}.
\end{align*}
Using again that $\sum_{j=1}^{K_{r}}a_{r,j}\geq\sum
_{i=1}^{K_{s}}a_{s,i}$ for $r\leq s$,
\begin{align*}
&  \sum_{r=1}^{R-1}\sum_{\substack{s=r+1\\B_{s}\text{\thinspace far
from\thinspace}B_{r}}}^{R}\sum_{j=1}^{K_{r}}a_{r,j}\sum_{i=1}^{K_{s}}%
a_{s,i}\left(  X^{1/d}d\left(  x_{r,j},x_{s,i}\right)  \right)  ^{-M}\\
= &  \sum_{r=1}^{R-1}\sum_{j=1}^{K_{r}}a_{r,j}\sum_{\ell=0}^{\infty}%
\sum_{\substack{s>r:\\2^{\ell-1}c_{3}Y^{-1/d}\leq d\left(  z_{r},z_{s}\right)
\leq2^{\ell}c_{3}Y^{-1/d}}}\sum_{i=1}^{K_{s}}a_{s,i}\left(  X^{1/d}d\left(
x_{r,j},x_{s,i}\right)  \right)  ^{-M}\\
\leq &  C\sum_{r=1}^{R-1}\sum_{j=1}^{K_{r}}a_{r,j}\sum_{\ell=0}^{\infty
}2^{-\ell M}\sum_{\substack{s>r:\\d\left(  z_{r},z_{s}\right)  \leq2^{\ell}c_{3}Y^{-1/d}}}\sum_{i=1}^{K_{s}}a_{s,i}\\
\leq &  C\sum_{r=1}^{R-1}\sum_{j=1}^{K_{r}}a_{r,j}\sum_{\ell=0}^{\infty
}2^{-\ell M}\frac{\left(  2^{\ell}Y^{-1/d}\right)  ^{d}}{Y^{-1}}\sum
_{j=1}^{K_{r}}a_{r,j}\\
\leq &  C\sum_{r=1}^{R-1}\left(  \sum_{j=1}^{K_{r}}a_{r,j}\right)  ^{2}%
\sum_{\ell=0}^{\infty}2^{-\ell\left(  M-d\right)  }\leq C\sum_{r=1}%
^{R-1}\left(  \sum_{j=1}^{K_{r}}a_{r,j}\right)  ^{2}.
\end{align*}
\end{proof}

\end{document}